\pdfoutput=1
\documentclass[a4paper]{amsart}
\usepackage[utf8]{inputenc}
\usepackage{dutchcal}
\usepackage{hyperref}
\usepackage{microtype}
\usepackage{amssymb}

\usepackage{tikz-cd}

\newtheorem*{mainthm}{Main Theorem}

\newtheorem{thm}{Theorem}[section]
\newtheorem{prop}[thm]{Proposition}
\newtheorem{lemma}[thm]{Lemma}
\newtheorem{cor}[thm]{Corollary}

\theoremstyle{definition}

\newtheorem*{notation}{Notation}

\newtheorem{remark}[thm]{Remark}
\newtheorem{exmpl}[thm]{Example}

\newcommand{\OX}{\ensuremath{\mathcal O_X}}
\newcommand{\OU}{\ensuremath{\mathcal O_U}}

\DeclareMathOperator{\QCoh}{QCoh}

\newcommand{\QCohX}{\ensuremath{\QCoh(X)}}
\newcommand{\QCohU}{\ensuremath{\QCoh(U)}}
\newcommand{\OXMod}{\ensuremath{\OX\text{\rm-Mod}}}

\newcommand{\RMod}{\ensuremath{R\text{\rm-Mod}}}

\newcommand{\res}{\mathrm{res}}

\newcommand{\qc}{\mathrm{qc}}
\newcommand{\op}{\mathrm{op}}

\newcommand{\ito}{\hookrightarrow}
\newcommand{\G}{\mbc G}

\newcommand{\GU}{\mbc G_U}
\newcommand{\GW}{\mbc G_W}

\DeclareMathOperator{\Hom}{Hom}

\DeclareMathOperator{\Hh}{H}
\DeclareMathOperator{\Spec}{Spec}
\newcommand{\HOM}{\operatorname{\mathbcal{Hom}}}
\newcommand{\HOMqc}{\operatorname{\mathbcal{Hom}}^{\qc}}

\let\mc\mathcal
\let\mbc\mathbcal

\begin{document}
\title{On flat generators and Matlis duality for quasicoherent sheaves}
\author{Alexander Sl\'avik}
\address{Department of Algebra, Charles University, Faculty of Mathematics and Physics, Sokolovsk\'a 83, 186\,75 Prague 8, Czech Republic}
\email{slavik.alexander@seznam.cz}
\author{Jan \v S\v tov\'\i\v cek}
\address{Department of Algebra, Charles University, Faculty of Mathematics and Physics, Sokolovsk\'a 83, 186\,75 Prague 8, Czech Republic}
\email{stovicek@karlin.mff.cuni.cz}
\thanks{Both authors were supported from the grant GA \v CR 17-23112S of the Czech Science Foundation. The first author was also supported from the grant SVV-2017-260456 of the SVV project and from the grant UNCE/SCI/022 of the Charles University Research Centre.}

\begin{abstract}
We show that for a quasicompact quasiseparated scheme $X$, the following assertions are equivalent: (1) the category \QCohX\ of all quasicoherent sheaves on $X$ has a flat generator; (2) for every injective object $\mc E$ of \QCohX, the internal hom functor into $\mc E$ is exact; (3) the scheme $X$ is semiseparated.
\end{abstract}


\maketitle

\section{Introduction}

Let $X$ be a scheme. It is well-known that unless the scheme is affine, the category $\QCohX$ of all quasicoherent sheaves on $X$ does not usually have enough projective objects.
In fact, a quasiprojective scheme over a field is affine if and only if  $\QCohX$ has enough projective objects, \cite[Theorem 1.1]{Kan}, and a direct proof that $\QCohX$ has no non-zero projective objects if $X$ is the projective line over a field can be found in \cite[Corollary 2.3]{EEG}.

Such an issue is often fixed using some flat objects; recall that a quasicoherent sheaf $\mc M$ is called \emph{flat} if for any open affine set $U \subseteq X$, the $\OX(U)$-module $\mc M(U)$ is flat. Murfet in his thesis \cite{M} showed that for $X$ quasicompact and semiseparated, every quasicoherent sheaf is a quotient of a flat one
(recall that a scheme is called \emph{semiseparated} if the intersection of any two open affine sets is affine; this differs from the original definition from \cite{TT}, but is shown to be equivalent in \cite{AJPV}).
A short proof of the same fact, attributed to Neeman, can be found in \cite[Appendix A]{EP}, which was (under the same assumptions) later improved by Positselski \cite[Lemma 4.1.1]{P} by showing that so-called very flat sheaves are sufficient for this job. These results can be rephrased that the category $\QCohX$ has a flat generator.

It was hoped for a long time that the existence of a flat generator can be extended at least to the case of quasicompact \emph{quasiseparated} schemes (i.e.\ those for which the intersection of any two open affine sets is quasicompact), which encompass a considerably wider class of ``natural'' examples arising in algebraic geometry, while being an assumption rather pleasant to work with.
However, our results show that for quasicompact quasiseparated schemes, semiseparatedness is in fact \emph{necessary} for the existence of enough flat quasicoherent sheaves.

In this context, we note that is has been already known that semiseparatedness is necessary for the existence of a generating set consisting of vector bundles. This is a consequence of much more involved structure theorems for stacks, see~\cite[Proposition 1.3]{Tot} and \cite[Theorem 1.1(iii)]{Gr}. Here we present a stronger version of that consequence with a much simpler proof.

\medskip

A question closely related to the existence of a flat generator turns out to be the exactness of the Matlis duality functor. If $R$ is a commutative ring and $E$ an injective cogenerator of the category \RMod, the Matlis duality functor $\Hom_R(-, E)\colon \RMod^\op \to \RMod$ has been considered on numerous occasions in the literature and one of its fundamental properties is that it is exact.

If $X$ is a possibly non-affine scheme, we can consider an analogous duality of the category $\QCohX$ of quasi-coherent sheaves. Namely, $\QCohX$ has an internal hom functor $\HOMqc$ which is right adjoint to the usual tensoring of sheaves of $\OX$-modules, and we can consider the functor $\HOMqc(-, \mc E)\colon \QCohX^\op \to \QCohX$ for an injective cogenerator $\mc E\in\QCohX$. For a simple formal reason which we discuss below, $\HOMqc(-, \mc E)$ is exact provided that $\QCohX$ has a flat generator, so in particular if $X$ is quasicompact and semiseparated. Perhaps somewhat surprisingly, we prove that for quasicompact quasiseparated schemes, semiseparatedness is again a \emph{necessary} condition for the exactness.

To summarize, our main result reads as follows:

\begin{mainthm}[see Theorems~\ref{no-flat-quotient} and~\ref{hom-not-exact}]
Let $X$ be a quasicompact and quasiseparated scheme. Then the following assertions are equivalent:
\begin{enumerate}
\item the category\/ \QCohX\ of all quasicoherent sheaves on $X$ has a flat generator;
\item for every injective object $\mc E$ of\/ \QCohX, the contravariant internal hom functor $\HOMqc(-, \mc E)$ is exact;
\item the scheme $X$ is semiseparated.
\end{enumerate}
\end{mainthm}


The paper is organized as follows.
In Section \ref{section-no-flat-generator}, we give a direct proof that for a non-semiseparated scheme $X$, the category \QCohX\ does not have a flat generator. The proof is rather constructive, producing a quasicoherent sheaf which is not a quotient of a flat one. Section \ref{section-internal-hom-exactness} then provides the characterization of semiseparated schemes using the exactness of the internal hom $\HOMqc(-, \mc E)$ for every injective quasicoherent sheaf $\mc E$. This in fact gives another, less explicit proof of the results of Section \ref{section-no-flat-generator}.

\begin{notation}
If $R$ is a commutative ring and $M$ an $R$-module, then by $\tilde M$ we denote the quasicoherent sheaf on $\Spec R$ with $M$ as the module of global sections. If the formula describing the module is too long and the tilde would not be wide enough, we use the notation like $M^\sim$.

If $U$ is an open subset of a scheme $X$, which is usually clear from the context, then $\iota_U \colon U \to X$ denotes the inclusion and $\iota_{U,*}$ the direct image functor. Since we are dealing only with quasicompact open sets, $\iota_U$ is a quasicompact and quasiseparated map, hence $\iota_{U,*}$ sends quasicoherent sheaves to quasicoherent sheaves by \cite[03M9]{Stacks}.
\end{notation}

\section{Non-existence of flat generators}\label{section-no-flat-generator}

Let $\mc M$ be a quasicoherent sheaf on an affine scheme $X$. If $U$ is an open affine subset of $X$ then it is a part of the very definition of a quasicoherent sheaf on $X$ that the map $\mc M(X) \otimes_{\OX(X)} \OX(U) \to \mc M(U)$ is an isomorphism of $\OX(U)$-modules. If $U$ is not affine, this may not be the case. However, more can be said for flat sheaves:

\begin{lemma}\label{flat-lemma}
Let $U$ be a quasicompact open subset of an affine scheme $X$ and $\mc F$ a flat quasicoherent sheaf on $X$. Then
\[ \mc F(U) \cong \mc F(X) \otimes_{\OX(X)} \OX(U), \]
where the isomorphism is obtained by tensoring the restriction map $\mc F(X) \to \mc F(U)$ with $\OX(U)$.
\end{lemma}
\begin{proof}
We may assume that $X = \Spec R$. The assertion is clearly true for the structure sheaf and all its finite direct sums (i.e.\ all finite rank free $R$-modules). Since $U$ is a quasicompact open subset of an affine scheme, it is also quasiseparated and the functor of sections over $U$ commutes with direct limits \cite[009F]{Stacks}. By the Govorov-Lazard Theorem \cite[058G]{Stacks}, any flat $R$-module is the direct limit of finite rank free modules, and since tensor product commutes with colimits, the desired property holds for all flat modules.
\end{proof}

\begin{thm}\label{no-flat-quotient}
Let $X$ be a quasicompact quasiseparated scheme. Then $X$ is semiseparated if and only if for each quasicoherent sheaf on $X$ is a quotient of a flat quasicoherent sheaf.
\end{thm}
\begin{proof}
If $X$ is semiseparated, then the assertion holds by the results mentioned in the introduction.

If $X$ is not semiseparated, let $U$, $V$ be two open affine subsets of $X$ such that the intersection $W = U \cap V$ is \emph{not} affine. Since $X$ is quasiseparated, $W$ is quasicompact; therefore, there are sections $f_1, \dots, f_n \in \OX(U)$ such that $W = U_{f_1} \cup \cdots \cup U_{f_n}$, where $U_f$ denotes the distinguished open subset of the affine subscheme $U$ where $f$ does not vanish. Denote by $I$ the ideal of $\OX(U)$ generated by $f_1, \dots, f_n$ and $\mc I = \iota_{U,*}(\tilde I)$ the direct image of $\tilde I$ with respect to the inclusion $\iota_U$.

Since $\mc I(U_{f_i}) = \OX(U_{f_i})$ for each $i = 1, \dots, n$, the sheaf axiom implies that $\mc I(W) = \OX(W)$.  On the other hand, by \cite[Chapter II, Exercise 2.17(b)]{H}, the restrictions of $f_1, \dots, f_n$ to $W$ do not generate the unit ideal of the ring $\OX(W)$.

Assume that there is a flat quasicoherent sheaf $\mc F$ and an epimorphism $f \colon \mc F \to \mc I$.
We have a commutative square
\[ \begin{tikzcd}
\mc F(V) \ar[r,"\res^{\mc F}_{WV}"] \ar[d,twoheadrightarrow,"f(V)"] & \mc F(W) \ar[d,twoheadrightarrow,"f(W)"] & \mc F(U) \ar[l,"\res^{\mc F}_{WU}"'] \ar[d,twoheadrightarrow,"f(U)"] \\
\mc I(V) \ar[r,equal] & \mc I(W) & \mc I(U) \ar[l,"\res^{\mc I}_{WU}"']
\end{tikzcd} \]
with the outer vertical arrows being epimorphisms due to $U$, $V$ being affine and the middle arrow due to the commutativity of the left-hand square.

Let us apply the functor $- \otimes_{\OX(U)} \OX(W)$ (shown simply as $- \otimes \OX(W)$ in the diagram) to the right-hand square, obtaining
\[ \begin{tikzcd}[column sep=7.2em]
\mc F(W) \ar[d,twoheadrightarrow,"f(W)"] & \mc F(U)\otimes \OX(W) \ar[l,"\res^{\mc F}_{WU}\otimes \OX(W)"'] \ar[d,twoheadrightarrow,"f(U)\otimes \OX(W)"] \\
\mc I(W) & \mc I(U)\otimes \OX(W) \ar[l,"\res^{\mc I}_{WU}\otimes \OX(W)"']
\end{tikzcd} \]
Since $\mc F$ is flat, by Lemma \ref{flat-lemma}, the top arrow is an isomorphism. Consequently, the composition gives an epimorphism $\mc F(U)\otimes_{\OX(U)} \OX(W) \to \mc I(W)$. However, the bottom arrow $\res^{\mc I}_{WU}\otimes_{\OX(U)} \OX(W)$ is not an epimorphism: As $\mc I(U) = I$, this map factors as
\[ I \otimes_{\OX(U)} \OX(W) \longrightarrow I \OX(W) \longrightarrow \mc I(W), \]
and as observed above, the image is a proper submodule of $\mc I(W) = \OX(W)$.
This is a contradiction and, hence, the quasicoherent sheaf $\mc I$ cannot be a quotient of a flat quasicoherent sheaf.
\end{proof}

\begin{exmpl} \label{no-flat-double-origin}
Let $k$ be any field.
A handy (and possibly the easiest) example of a non-semiseparated scheme $X$ is the \emph{plane with double origin}, obtained by gluing two copies of $\Spec k[x, y]$ along the \emph{punctured plane}, i.e.\ the non-affine open subset containing everything except the maximal ideal $(x, y)$. It may be illuminating to trace the proof of Theorem \ref{no-flat-quotient} in this particular case.

Let $R = k[x, y]$ for brevity. Then $U$, $V$ are the two copies of $\Spec R$ and $W$ is the punctured plane. Then $I = (x, y)$, as $W$ can be covered by the two distinguished open subsets $U_x$, $U_y$. The resulting sheaf $\mc I$ then satisfies $\mc I(U) = I$, $\mc I(V) = \mc I(W) = R$. Since, by Lemma \ref{flat-lemma}, for any flat sheaf on $X$, both restrictions from $U$ and $V$ to $W$ are the identity morphisms, it is easy to see that the image of any map from a flat sheaf is contained in $I$ on all three open sets.
\end{exmpl}

\section{Exactness of the internal Hom}\label{section-internal-hom-exactness}

For any scheme $X$, the category \QCohX\ has a closed symmetric monoidal structure given by the usual sheaf tensor product $\otimes$ together with its right adjoint, which we denote by $\HOMqc$. This bifunctor is just the usual sheaf hom composed with the coherator functor \cite[Proposition 6.15]{M}. In this section we investigate the exactness of the contravariant functor $\HOMqc(-, \mc E)$, where $\mc E$ is an injective object of \QCohX.

If $(\mbc G,\otimes)$ is a general symmetric monoidal category, we call an object $F\in\mbc G$ flat if the functor $F\otimes-\colon \mbc G \to \mbc G$ is exact. This is well-known to be consistent with the previous definition of flatness for $\mbc G = \QCohX$.

We start with a general observation:

\begin{lemma}\label{enriched-injective}
Let $\mbc G$ be a symmetric closed monoidal abelian category with the internal hom denoted\/ $[-,-]$. Let $E$ be an injective object of $\mbc G$ and assume that $\mbc G$ has a flat generator $G$. Then the functor
\[ [-,E] \colon \mbc G^\op \to \mbc G \]
is exact.
\end{lemma}
\begin{proof}
The internal hom, being a right adjoint, is always left-exact. Having a monomorphism $A \ito B$ in $\G$, to test that $[B,E] \to [A,E]$ is an epimorphism, we can check it against the exactness-reflecting functor $\Hom_\G(G,-)$,
\[ \Hom_\G(G, [B,E]) \to \Hom_\G(G, [A,E]), \]
which, using the adjunction, is surjective if and only if
\[ \Hom_\G(G \otimes B, E) \to \Hom_\G(G \otimes A, E) \]
is, where $\otimes$ denotes the tensor product in $\mbc G$. Since $G$ is flat, $G \otimes A \to G \otimes B$ is a monomorphism, and injectivity of $E$ implies that the map in question is indeed surjective.
\end{proof}

\begin{exmpl}
Let $\mbc G$ be the category of chain complexes of vector spaces over a field. This is a Grothendieck category where the injective objects are precisely the contractible complexes. However, the internal hom is exact for any arguments, hence there is in general no converse to Lemma \ref{enriched-injective} in the sense that if $[-,E]$ is exact, then $E$ is injective.
\end{exmpl}

\begin{cor}\label{semiseparated-hom-exact}
Let $X$ be a quasicompact and semiseparated scheme. Then for every injective $\mc E \in \QCohX$, the functor
\[ \HOMqc(-,\mc E)\colon \QCohX^\op \to \QCohX \]
is exact.
\end{cor}
\begin{proof}
As pointed out in the introduction, the category \QCohX\ has a flat generator whenever $X$ is quasicompact semiseparated, hence Lemma \ref{enriched-injective} applies.
\end{proof}

For the sake of completeness, we also record the following result.

\begin{prop}
Let $X$ be a scheme and $\mc E \in \QCohX$ an injective quasicoherent sheaf such that $\mc E$ is also an injective object of the category \OXMod\ of all sheaves of $\OX$-modules. Then the functor $\HOMqc(-,\mc E)$ is exact on short exact sequences of sheaves of finite type.
\end{prop}
\begin{proof}
By \cite[05NI]{Stacks}, the category \OXMod\ has a family of flat generators formed by the extensions by zero of the restrictions $\OX|_U$, where $U \subseteq X$ is an open set. Therefore, by the assumption on $\mc E$ and Lemma \ref{enriched-injective}, the usual sheaf hom $\HOM(-,\mc E)$ is exact on \OXMod. Finally, by \cite[01LA]{Stacks}, $\HOM(\mc A, \mc E)$ is quasicoherent for every $\mc A$ of finite type, hence it coincides with $\HOMqc(\mc A, \mc E)$ and we are done.
\end{proof}

Note that the assumption of $\mc E$ being injective in \OXMod\ is satisfied e.g.\ whenever $X$ is locally Noetherian, \cite[\S II.7]{H2}. This shows that to produce a counterexample to Corollary \ref{semiseparated-hom-exact} with $X$ locally Noetherian non-semiseparated, one has to work with sheaves not of finite type.

Furthermore, in the locally Noetherian case, the proof shows that the sheaf hom into an injective sheaf is exact, so it is the coherator that is ``responsible'' for the failure of exactness in general.

Next we will need the following enriched version of the adjunction between the restriction to an open set and the direct image functor.

\begin{lemma}\label{hom-to-direct-image}
Let $X$ be a quasicompact and quasiseparated scheme and $U \subseteq X$ an open quasicompact subset. Then, for every $\mc M \in \QCohX$ and $\mc N \in \QCohU$ we have a natural isomorphism
\[ 
\HOMqc(\mc M, \iota_{U,*}(\mc N))
\cong \iota_{U,*}\bigl(\HOMqc(\mc M|_U, \mc N)\bigr).
\]
If $U$ is affine, this is further isomorphic to
\[ \iota_{U,*}\bigl(\Hom_{\QCohU}(\mc M|_U, \mc N)^\sim\bigr)
\quad \text{and} \quad
\iota_{U,*}\bigl(\Hom_{\OX(U)}(\mc M(U), \mc N(U))^\sim\bigr).
\]
\end{lemma}
\begin{proof}
Since restriction to $U$ commutes with the tensor product, the following two functors $\QCohX \to \QCohU$ are naturally isomorphic:
\[ \mc M|_U \otimes (-)|_U \cong (\mc M \otimes -)|_U. \]
Both functors are compositions of left adjoints---the restriction and the tensor product. Hence composing the corresponing right adjoints produces the isomorphism from the statement.

The further isomorphisms in case $U$ is affine are direct consequences of the fact that quasicoherent sheaves over $U$ are determined by their modules of global sections.
\end{proof}

Let us recall further relevant definitions, which we are going to use: A subcategory of a Grothendieck category is called \emph{Giraud subcategory} if the inclusion functor has an exact left adjoint (cf.\ \cite[\S X.1]{Sten}). A subcategory of a locally finitely presented category with products (in the sense of~\cite{CB}) is \emph{definable} provided it is closed under direct products, direct limits, and pure subobjects.

\begin{lemma}\label{giraud-definable}
Let $R$ be a commutative ring, $X = \Spec R$ and $U$ a quasicompact open subset of $X$. Then the $R$-modules of the form $\tilde M(U)$, where $M \in \RMod$, form a Giraud, definable subcategory of $\RMod$, which we denote $\GU$. The inclusion functor $i\colon \GU \ito \RMod$ is exact if and only if $U$ is affine.
\end{lemma}
\begin{proof}
We have the following diagram of categories and functors:
\[ \begin{tikzcd}[column sep=large,row sep=large]
\QCohU \arrow[r, hook, "\iota_{U,*}"', bend right=10] \arrow[r, phantom, "\perp"] \arrow[d,leftrightarrow,"\cong"'] & \QCohX \arrow[d,leftrightarrow,"\cong"] \arrow [l, "(-)|_U"', bend right=10] \\
\GU \arrow[r, hook, "i"', bend right=10] \arrow[r, phantom, "\perp"] & \RMod \arrow [l, "\eta"', bend right=10, dashed]
\end{tikzcd} \]
The left-hand vertical equivalence follows from \cite[0EHM]{Stacks}, utilizing that $U$ is quasiaffine, and the right-hand one is the standard one; in both cases the passage from sheaves to modules is just taking the global sections. The (fully faithful) direct image functor $\iota_{U,*}$ identifies \QCohU\ with a full subcategory of \QCohX\ with the restriction to $U$ being the exact left adjoint. We define $\eta$ to be the composition of the sheaf restriction with the two vertical equivalences, hence $\eta$ is the exact left adjoint to the inclusion $i$. This shows that $\GU$ is a Giraud subcategory of \RMod.

Similarly, to show that $\GU$ is definable, we need to show that the essential image of the functor $\iota_{U,*}$ is a definable subcategory of \QCohX. Basically, one has to observe that \cite[Remark 4.6]{PS} generalizes to (possibly non-affine) quasicompact open subsets of $X$: As a right adjoint, $\iota_{U,*}$ commutes with direct products, and by \cite[Lemma B.6]{TT}, it commutes with direct limits. The closedness of the essential image under pure subobjects follows from \cite[Lemma 2.12]{PS} and the fact that by \cite[Lemma 1.4(2)]{PS}, categorical pure-exactness in \QCohX\ is inherited from the larger category of all sheaves of \OX-modules.

For the final claim, note that if $U$ is affine, then since $X$ is semiseparated, the functor $\iota_{U,*}$ is exact, which via the vertical equivalences implies the exactness of $i$. On the other hand, if $U$ is not affine, then by Serre's criterion \cite[01XF]{Stacks}, the sections over $U$, i.e.\ the composition of the left-hand equivalence with $i$, is not an exact functor $\QCohU \to \RMod$, therefore $i$ is not exact.
\end{proof}

\begin{thm}\label{hom-not-exact}
Let $X$ be a quasicompact quasiseparated scheme. Then $X$ is semiseparated if and only if for each $\mc E \in \QCohX$ injective, the contravariant functor $\HOMqc(-,\mc E)$ is exact.
\end{thm}
\begin{proof}
If $X$ is semiseparated, then Lemma \ref{semiseparated-hom-exact} applies.

Assume that $X$ is not semiseparated; then there are two open affine subsets $U$, $V$ of $X$ such that the intersection $W = U \cap V$ is not affine. Since $X$ is quasiseparated, $W$ is quasicompact. By Serre's criterion \cite[01XF]{Stacks}, there is $\mc A' \in \QCoh(W)$ satisfying $\Hh^1(W, \mc A') \neq 0$. Since $\QCoh(W)$ is a Grothendieck category, there is an embedding $\mc A' \ito \mc B'$ with $\mc B' \in \QCoh(W)$ injective; in particular, $\Hh^1(W, \mc B') = 0$. Let $\mc C' \in \QCoh(W)$ be the cokernel of this embedding; hence we have a short exact sequence $0 \to \mc A' \to \mc B' \to \mc C' \to 0$ with non-exact sequence of sections over $W$.

Put $\mc A = \iota_{W,*}(\mc A')$, $\mc B = \iota_{W,*}(\mc B')$. The direct image functor is left exact; let $\mc C$ be the cokernel of $\mc A \ito \mc B$. Since the restriction to $W$ is an exact functor and $\mc A|_W = \mc A'$, $\mc B|_W = \mc B'$, it follows that $\mc C|_W = \mc C'$. We have obtained a short exact sequence $0 \to \mc A \to \mc B \to \mc C \to 0$ in \QCohX\ with non-exact sequence of sections over $W$.

Let $R = \OX(U)$ and let $E$ be an injective cogenerator of \RMod. Further, put $\mc E = \iota_{U,*}(\tilde E)$. Since $\iota_{U,*}$ is a right adjoint to an exact functor, it preserves injectives, hence $\mc E$ is an injective object of \QCohX. For $\mc M \in \QCohX$, denote by $\mc M^+$ the sheaf $\HOMqc(\mc M, \mc E)$.
We are going to show that the sequence
\[ 0 \to \mc A^{++} \to \mc B^{++} \to \mc C^{++} \to 0 \tag{$++$} \]
cannot be exact by showing that the sections over the open affine set $V$ are not exact.

Denote further $M^* = \Hom_R(M, E)$ for $M \in \RMod$.
By Lemma \ref{hom-to-direct-image}, we have
\begin{multline*}
\mc M^+
= \HOMqc(\mc M, \mc E)
= \iota_{U,*}\bigl(\Hom_R(\mc M(U), \mc E(U))^\sim\bigr) = \\
= \iota_{U,*}\bigl(\Hom_R(\mc M(U), E)^\sim\bigr)
= \iota_{U,*}\bigl(\widetilde{\mc M(U)^*}\bigr),
\end{multline*}
therefore
\begin{multline*}
\mc M^{++}
= \HOMqc(\mc M^+, \mc E)
= \iota_{U,*}\bigl(\Hom_R(\mc M^+(U), E)^\sim\bigr) = \\
= \iota_{U,*}\bigl(\Hom_R(\mc M(U)^*, E)^\sim\bigr)
= \iota_{U,*}\bigl(\widetilde{\mc M(U)^{**}}\bigr)
\end{multline*}
for every $\mc M \in \QCohX$; in particular, $\mc M^{++}(V) = \mc M^{++}(W)$. Put $A = \mc A(U)$, $B = \mc B(U)$, and $C = \mc C(U)$. By the construction, $A = \mc A(W)$, $B = \mc B(W)$, so $A, B \in \GW$. On the other hand $C \subsetneq \mc C(W)$, hence $C \notin \GW$. Indeed, $C$ is the cokernel of $A\ito B$ in $\RMod$ and $\mc C(W)$ is the cokernel of $A\ito B$ in $\GW$, so $C\in \GW$ would imply $C=\mc C(W)$.

Since by Lemma \ref{giraud-definable} $\GW$ is a definable subcategory of \RMod, \cite[Corollary 3.4.21]{Pr} implies that $A^{**}, B^{**} \in \GW$, but $C^{**} \notin \GW$. In other words,
\begin{align*}
\mc A^{++}(V) &= \mc A^{++}(W) = \mc A^{++}(U) = A^{**},\\
\mc B^{++}(V) &= \mc B^{++}(W) = \mc B^{++}(U) = B^{**},
\end{align*}
but
\[ \mc C^{++}(V) = \mc C^{++}(W) \not\cong \mc C^{++}(U) = C^{**}. \]
However, as $(-)^*$ is an exact contravariant functor on \RMod, the sequence $0 \to A^{**} \to B^{**} \to C^{**} \to 0$ is exact. This shows that the sequence $(++)$ is not exact after passing to sections over $V$ as desired.
\end{proof}

\begin{remark}
Note that we have re-proved Theorem \ref{no-flat-quotient}: If $X$ is quasicompact and quasiseparated, but not semiseparated, then by Theorem \ref{hom-not-exact}, there is an injective $\mc E \in \QCohX$ such that $\HOMqc(-, \mc E)$ is not exact; by Lemma \ref{enriched-injective}, this means that the category \QCohX\ cannot have a flat generator.
\end{remark}

\begin{exmpl}
As in Example~\ref{no-flat-double-origin}, it may be instructive to consider the plane with double origin. We again put $R = k[x, y]$, the open subsets $U$, $V$ will be the two copies of $\Spec R$ and $W$ will be the punctured plane.

The following example was suggested to us by Leonid Positselski. The sheafification $0 \to \OU \to \OU \to \mc M \to 0$ of the short exact sequence $0 \to k[x, y] \overset{y}\to k[x, y] \to k[x] \to 0$ of $R$-modules has non-exact sections on $W$. Indeed, a direct computation using the sheaf axiom reveals that the sections on $W$ form only the left exact sequence
\[ 0 \to k[x, y] \overset{y}\to k[x, y] \to k[x^{\pm1}]. \tag{$*$} \]

If we take the sheafification of the same exact sequence of $R$-modules in $\QCoh(V)$ and glue it together with the sequence in $\QCohU$, we obtain a short exact sequence $0 \to \mc A \to \mc B \to \mc C \to 0$ of quasicoherent sheaves on $X=U\cup V$ whose sections on $W$ look like~$(*)$.

Let $E\in\RMod$ be the injective envelope of the simple $R$-module $R/(x, y)$.
The arguments in the proof of Theorem~\ref{hom-not-exact} show that the sections of the double dual $(++)$, where $(-)^+ = \HOMqc(-,\mc E)$ for $\mc E=\iota_{U,*}(\tilde E)$, on the affine open set $V$ are of the form
\[ 0 \to k[[x, y]] \overset{y}\to k[[x, y]] \to k(\!(x)\!), \]
which is again a left, but not right exact sequence.

Although, for the sake of simplicity, this example does not fully follow the proof of Theorem~\ref{hom-not-exact} (the sheaf $\mc B|_W$ in not injective in $\QCoh(W)$ and $E$ is not an injective cogenerator of $\RMod$), it is still sufficient to explicitly illustrate the non-exactness of $\HOMqc(-,\mc E)$ on $\QCohX$.
\end{exmpl}

\begin{remark}
Let $R$ be a commutative ring and $E$ an injective cogenerator of the category \RMod. The contravariant functor $\Hom_R(-, E)$ has found many applications in the model theory of modules and its generalizations (cf.\ \cite[1.3.3]{Pr}; this also works in greater generality over non-commutative rings). This has led to a natural generalization to symmetric closed monoidal Grothendieck categories; in particular, in \cite{EEO}, the functor $\HOMqc(-, \mc E)$, where $\mc E$ is an injective cogenerator of \QCohX, has been used to investigate the properties of ``geometric'' purity.

Theorem \ref{hom-not-exact} shows that this ``dual'' is, perhaps surprisingly, not exact for non-semiseparated schemes. However, it turns out that this is not really an obstacle to using this functor for investigating purity in the same way as in the classical situation, cf.\ \cite[Proposition 4.5]{EEO}. Furthermore, this functor at least reflects exactness, as the next proposition shows.
\end{remark}

\begin{prop}
Let $\mc E$ be an injective cogenerator of\/ \QCohX. Then the functor $\HOMqc(-, \mc E)$ reflects exactness.
\end{prop}
\begin{proof}
Assume that
\[ 0 \to \HOMqc(\mc C, \mc E) \to \HOMqc(\mc B, \mc E) \to \HOMqc(\mc A, \mc E) \to 0 \]
is exact. By \cite[Proposition 4.4 \& Lemma 4.7]{EEO}, all the terms are (geometrically) pure-injective. In particular, by \cite[Lemma 4.15]{PS}, $\HOMqc(\mc C, \mc E)$ has zero cohomology, therefore taking global sections produces a short exact sequence
\[ 0 \to \Hom_{\QCohX}(\mc C, \mc E) \to \Hom_{\QCohX}(\mc B, \mc E) \to \Hom_{\QCohX}(\mc A, \mc E) \to 0. \]
As $\mc E$ is a cogenerator, this implies that $0 \to \mc A \to \mc B \to \mc C \to 0$ is exact.
\end{proof}


\begin{thebibliography}{33}
\bibitem{AJPV} Leovigildo Alonso, Ana Jeremias, Marta Perez, Maria J. Vale, \emph{The derived category of quasi-coherent sheaves and axiomatic stable homotopy}. Advances in Mathematics 218, \#4, p.~1224--1252, 2008.
\bibitem{CB}  William Crawley-Boevey,
\emph{Locally finitely presented additive categories}.
Comm. Algebra 22, \#5, p. 1641--1674, 1994.
\bibitem{EEG} Edgar Enochs, Sergio Estrada, Juan R. Garc\'\i a-Rozas,  Luis Oyonarte, \emph{Flat
cotorsion quasi-coherent sheaves. Applications}. Algebras and Representation Theory 7, \#4, p.\ 441--456, 2004.
\bibitem{EEO}
Edgar Enochs, Sergio Estrada, Sinem Odabaş\i, \textit{Pure Injective and Absolutely Pure Sheaves}. Proc. Edinburgh Math. Soc. 59, \#3, p.\ 623--640, 2016.
\bibitem{EP} Alexander I. Efimov, Leonid Positselski,  \emph{Coherent analogues of matrix factorizations and relative singularity categories}. Algebra and Number Theory 9, \#5, p.\ 1159--1292, 2015.
\bibitem{Gr} Philipp Gross,
\emph{Tensor generators on schemes and stacks}.
Algebr. Geom. 4, \#4, p. 501--522, 2017.
\bibitem{H} Robin Hartshorne, \emph{Algebraic geometry}. Graduate Texts in Mathematics, 52., Springer, 1977.
\bibitem{H2} Robin Hartshorne, \emph{Residues and duality}. Springer, 1966.
\bibitem{Kan} Ryo Kanda, \emph{Non-exactness of direct products of quasi-coherent sheaves}.
\texttt{arXiv:1810.08752v1}
\bibitem{M} Daniel Murfet, \emph{The Mock Homotopy Category of Projectives and Grothendieck Duality}. Australian National University, September 2007, PhD Thesis. Available at \url{http://therisingsea.org/thesis.pdf}.
\bibitem{P} Leonid Positselski, \emph{Contraherent cosheaves}. \texttt{arXiv:1209.2995v6}
\bibitem{Pr} Mike Prest, \emph{Purity, spectra and localization}. Encyclopedia of Mathematics and its Applications, 121., Cambridge University Press, 2009.
\bibitem{PS} Mike Prest, Alexander Sl\'avik, \emph{Purity in categories of sheaves}. \texttt{arXiv:1809.08981v2}
\bibitem{Sten} Bo Stenstr\"{o}m,
\emph{Rings of quotients}.
Die Grundlehren der Mathematischen Wissenschaften, Band 217. Springer-Verlag, New York-Heidelberg, 1975.
\bibitem{TT} Robert W. Thomason, Thomas Trobaugh, \emph{Higher Algebraic K-Theory of Schemes and of Derived Categories}. The Grothendieck Festschrift vol. 3, p.\ 247--435, Birkhäuser, 1990.
\bibitem{Tot} Burt Totaro,
\emph{The resolution property for schemes and stacks}.
J. reine angew. Math. 577, p. 1--22, 2004.
\bibitem{Stacks} The Stacks project authors, \emph{Stacks project}. Available at \url{https://stacks.math.columbia.edu/}

\end{thebibliography}
\end{document}